\documentclass[12pt, oneside]{amsart}
\usepackage{amsthm}
\usepackage[margin=1in, marginparwidth=2cm]{geometry}
\usepackage{bbm}
\usepackage{mathrsfs}
\usepackage{bm}
\usepackage{color}
\usepackage{enumitem}
\usepackage{tikz}
\usepackage{amssymb,amsthm,amsmath}
\usepackage[numbers,sort&compress]{natbib}
\usepackage{graphicx}
\usepackage{xparse}
\usepackage[
	colorlinks,
	linkcolor={black},
	citecolor={black},
	urlcolor={black}
]{hyperref}
\usepackage{verbatim}
\usepackage{todonotes}

\numberwithin{equation}{section}

\newtheorem{theorem}{Theorem}[section]
\newtheorem{lemma}[theorem]{Lemma}
\newtheorem{prop}[theorem]{Proposition}
\newtheorem{coro}[theorem]{Corollary}

\theoremstyle{definition}
\newtheorem{definition}[theorem]{Definition}

\newtheorem{remark}[theorem]{Remark}

\theoremstyle{plain}

\newcommand{\bbC}{{\mathbbm{C}}}

\newcommand{\bbN}{{\mathbbm{N}}}

\newcommand{\bbR}{{\mathbbm{R}}}
\newcommand{\bbT}{{\mathbbm{T}}}
\newcommand{\bbZ}{{\mathbbm{Z}}}
\newcommand{\Z}{{\mathbbm{Z}}}
\newcommand{\C}{{\mathbbm{C}}}
\newcommand{\N}{{\mathbbm{N}}}

\newcommand{\calE}{{\mathcal{E}}}
\newcommand{\calG}{{\mathcal{G}}}

\newcommand{\calV}{{\mathcal{V}}}

\newcommand{\scrF}{{\mathscr{F}}}

\newcommand{\scrH}{{\mathscr{H}}}

\newcommand{\iop}{\mathbf{i}}

\newcommand{\tcp}{\widetilde{\mathcal{P}}}
\newcommand{\subscript}[2]{$#1 _ #2$}

\DeclareMathOperator{\diag}{{diag}}

\newcommand{\polym}{{\C[x_1,\ldots,x_m]}}
\newcommand{\laurm}{{\C[x_1^{\pm 1}, \ldots, x_m^{\pm 1}]}}

\definecolor{purple}{rgb}{.5,0,1}
\definecolor{green}{rgb}{0,.5,0}

\begin{document}
\title[Fermi Varieties of Periodic Graph Operators]{Algebraic Properties of the Fermi Variety for Periodic Graph Operators}
\author[J.\ Fillman]{Jake Fillman}
\email{fillman@txstate.edu}
\address{Department of Mathematics, Texas State University, San Marcos, TX 78666, USA}
\thanks{J.\ F.\ was supported in part by Simons Foundation Collaboration Grant \#711663 and NSF grant DMS--2213196. W.\ L.\  was  supported by NSF  DMS-2000345, DMS-2052572 and DMS-2246031. }

\author[W.\ Liu]{Wencai Liu}
\email{liuwencai1226@gmail.com; wencail@tamu.edu}
\address{Department of Mathematics, Texas A\&M University, College Station TX, 77843, USA.}

\author[R.\ Matos]{Rodrigo Matos}
\email{matosrod@tamu.edu}
\address{Department of Mathematics, Texas A\&M University, College Station TX, 77843, USA.}

\maketitle

\begin{abstract}
We present a method to estimate the number of irreducible components of the Fermi varieties of periodic Schr\"odinger operators on graphs in terms of suitable asymptotics. Our main theorem is an abstract bound for the number of irreducible components of Laurent polynomials in terms of such asymptotics.
We then show how the abstract bound implies irreducibility in many lattices of interest,
including examples with more than one vertex in the fundamental cell such as the Lieb lattice as well as certain models obtained by the process of graph decoration.  
\end{abstract}


\hypersetup{
	linkcolor={black!30!blue},
	citecolor={red},
	urlcolor={black!30!blue}
}

\section{Introduction} 
In recent decades, there  has been intense activity on the algebraic and analytic properties  of the Bloch and Fermi varieties of periodic Schr\"odinger operators. Such  properties are related to many problems of interest, including the absence of embedded eigenvalues in continuous settings \cite{kvcpde20} and discrete settings \cite{kv06cmp,LiuPreprint:Irreducibility}, the existence of eigenfunctions of unbounded support in locally perturbed periodic media \cite{shi1}, properties of spectral band functions \cite{LiuPreprint:Irreducibility},  inverse spectral problems \cite{liu2d,LiuPreprint:fermi, liufloquet23, GKTBook},  quantum ergodicity \cite{liuqua22,ms22},  flat bands \cite{BGM2022IMRN, KFSH2020CMP, KS2014JMAA, sabri2023flat}, and ballistic motion \cite{BdMS2022ballistic, DLY, F2021Ballistic}. For additional information and background see, e.g., the surveys \cite{Kuchment2016BAMS, kuchment2023analytic, liu2021topics}.

 The study of irreducibility of Fermi and Bloch varieties started about 30 years ago. B\"atting, Gieseker, Kn\"{o}rrer, and Trubowitz studied this question for periodic Schr\"odinger operators in dimensions $d\in\{2,3\}$ by compactification methods \cite{GKTBook, batcmh92, battig1988toroidal}. When  $d=2$,  the irreducibility of the Bloch  variety was  proved by B{\"a}ttig ~\cite{battig1988toroidal}. In ~\cite{GKTBook}, Gieseker, Kn\"orrer and Trubowitz proved that the Fermi variety is irreducible except for finitely many values of $\lambda$. When  $d=3$, B\"{a}ttig  proved that the Fermi variety is irreducible for every $\lambda$~\cite{batcmh92}.\par
 For continuous  periodic Schr\"odinger operators, Kn\"orrer and Trubowitz proved that  the Bloch variety is irreducible (modulo periodicity) when $d=2$ ~\cite{ktcmh90}. When the periodic  potential is separable, B{\"a}ttig, Kn\"orrer and Trubowitz proved that the Fermi variety at any level is irreducible (modulo periodicity) for $d=3$ ~\cite{bktcm91}.\par
 For further developments on these topics, see \cite{ls} for the case of planar periodic graphs and \cite{faust2023irreducibility} for a work which explores the connection with discrete geometry.
 
 Recently, one of the authors introduced a novel  approach  to study  the irreducibility of polynomials, obtaining several new results for periodic operators of the form $-\Delta+V$ on $\ell^2\left(\mathbb{Z}^d\right)$. In this case, Liu  proved that for $d=2$, the Fermi variety is irreducible at every energy level $\lambda$ except for, possibly, the average energy level. He also proved that when $d\geq 3$, the Fermi variety is irreducible for every level $\lambda$ \cite{LiuPreprint:Irreducibility}.
 In particular, for such operators it follows that the Bloch variety is irreducible in arbitrary dimension \cite{LiuPreprint:Irreducibility}. 
 The results in ~\cite{LiuPreprint:Irreducibility} provide a  complete proof for conjectures about the  irreducibility  of  Fermi and Bloch  varities in the discrete settings, as discussed in numerous articles~\cite{Kuchment2016BAMS,kvcpde20,GKTBook,ktcmh90,bktcm91,batcmh92}.

On a technical level, the approach in \cite{LiuPreprint:Irreducibility} can be divided into several steps. 
 \begin{enumerate}
     \item[1.] Show that the variety associated with every irreducible factor contains certain singular points in the Riemann sphere.
     \item[2.] Calculate the asymptotics (tangent cones after changing coordinates of variables) of the Fermi variety at those singular points.
     \item[3.] Prove that the asymptotics themselves are irreducible.
     \item[4.]  Use degree arguments to complete the proof. 
 \end{enumerate}

Our previous work \cite{FLM1} studied the Bloch variety associated to certain $\bbZ^d$-periodic graphs and Schr\"odinger operators given by the sum of a multiplication operator (the potential) and a Toeplitz operator (governing the site-to-site interactions). 
Following the approach outlined above, we introduced new ideas to study the Bloch varieties associated to such operators on single-vertex models, that is, $\bbZ^d$-periodic graphs in which the action is transitive on the vertices (i.e., there is a single vertex orbit under the $\bbZ^d$ action). 
Compared to the irreducibility of Fermi varieties, we were able to work with asymptotics that also depend on the spectral parameter $\lambda$ in a specific way. 
As a result we proved the irreducibility of Bloch variety for a large family of periodic graph operators~\cite{FLM1}.

The key idea  in both works \cite{LiuPreprint:Irreducibility,FLM1} is to find a way to reduce the problem to one in which the potential disappears. 
Indeed, the dependence of the full dispersion relation on the potential can be very delicate, so this is exactly the role of the asymptotic terms. 
In both cases, the approach is to expand the dispersion relation in terms of the symbol of the Toeplitz operator governing the site-to-site interactions. 
However, there are a couple of drawbacks to their results. 
First,  \cite{LiuPreprint:Irreducibility} only focuses on single-vertex models;
many models of physical interest such as graphene, the Lieb lattice and Kagome lattices and others do not fall into the single-vertex category.
Secondly, the arguments in \cite{FLM1} are only effective in studying irreducibility of the dispersion relation itself (equivalently, irreducibility of the \emph{Bloch variety}). 
However, the most interesting irreducibility results are for the dispersion relation with fixed energy, which corresponds to the \emph{Fermi variety}. 
The present paper thus develops a method to study the dispersion relation at fixed energy.
In particular, the technique here gives new results, even for the single-vertex case.

We now give context to the present manuscript and point out several difficulties inherent to the setting of many-vertex models which had to be overcome by us in order to implement the strategy outlined above in steps 1-4. We also comment on the importance of the models studied here for the mathematics and physics of periodic structures.

Our goal in this note is to present irreducibility criteria which can be applied to many-vertex models, that is, periodic lattices with more than one vertex in their fundamental cell.  Although the general scheme of the proof follows \cite{LiuPreprint:Irreducibility}, the framework and its implementation for specific examples are significantly more challenging. The many-vertex context is typically introduced to interpret physical systems which either contain many particles or allow for internal degrees of freedom such as spins, energy levels, and other physical parameters. Therefore, the scope of Hamiltonians which arise in this study, even when restricted to the $\mathbb{Z}^d$-periodic setting, is wide and also corresponds to a diverse range of physical phenomena.

One guiding example for the results presented here is the Lieb lattice, which is ubiquitous in nanostructures and exhibits exotic flat band structures. This lattice is of interest to ferromagnetism, superconductivity and fractional quantum Hall effects; see \cite{jianghuangliu, physicsnano} and references therein. A finite portion of the two dimensional Lieb lattice with three vertices in the fundamental cell is shown in Figure~\ref{fig:Liebd2} below. 

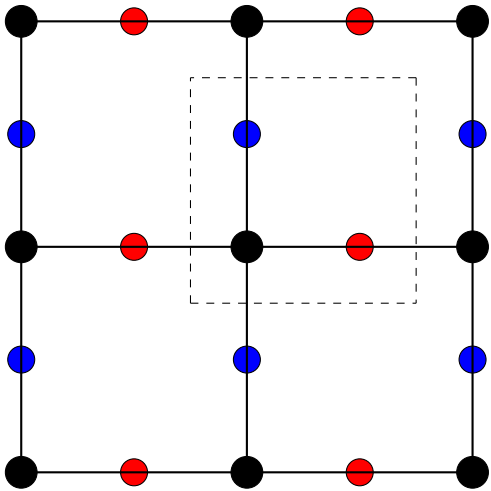
\begin{figure}[b!]

\begin{tikzpicture}[scale=3]
\def\width{2};
\def \length {2};
\foreach \m in {0,...,\width}
   \foreach \p in {0,...,\length}
        \draw[fill=black] (\m,\p) circle (2pt);
        \foreach \m in {0,...,\numexpr\width-1}
         \foreach \p in {0,...,\numexpr\length}
        \draw[fill=red] (\m+0.5,\p) circle (1.7pt);
        \foreach \m in {0,...,\numexpr\width}
         \foreach \p in {0,...,\numexpr\length-1}
        \draw[fill=blue] (\m,\p+0.5) circle (1.7pt);
\foreach \m in {0,...,\numexpr\width }
    \draw[thick] (\m,0) -- (\m , \width);
    \foreach \p in {0,...,\numexpr\width }
    \draw[thick] (\width,\p) -- (0, \p);
\draw[dashed] (.75,.75) -- (.75,1.75);
\draw[dashed] (.75,.75) -- (1.75,.75);
\draw[dashed] (1.75,1.75) -- (.75 ,1.75);
\draw[dashed] (1.75,1.75) -- (1.75 ,.75);

\end{tikzpicture}
\caption{A finite portion of the two dimensional Lieb lattice. A fundamental cell is show with dashed lines.}
\label{fig:Liebd2}
\end{figure}
Many-vertex models also naturally arise in the process of graph decoration. While this process may produce drastic effects on the spectrum of the underlying adjacency operator, such as the creation of spectral gaps \cite{AS}, we show below that in many examples of interest irreducibility is unaffected by decoration procedures.  In this way, we are able to present new irreducibility results for a class of examples of interest to mathematical physics. We now comment on some technical merits of the approach developed in this manuscript.\par
A first contribution of the present work is to present unified criteria for irreducibility. Indeed, the main results of Theorems \ref{thm1} and \ref{thm2} below include both types of examples discussed above and also recover known results for single-vertex models in which the action by $\Z^d$ has a single vertex orbit. In short, we show that in many situations of interest the Fermi variety is irreducible for all but finitely many energies.
The criteria obtained here reduces the question of counting irreducible factors of the Fermi variety to counting factors of suitable asymptotics, which is achieved in Theorems \ref{thm1} and \ref{thm2}.

We also emphasize that our main results of Theorems \ref{thm1} and \ref{thm2} are abstract and possibly of independent interest. Moreover, these results are not only applicable to the study of irreducibility but, rather, they bound the number of irreducible components of a polynomial in terms of the number of components of two of its asymptotics whenever the appropriate conditions are met. Within the realm of applications to irreducibility of the Fermi variety of periodic operators, the operators considered here are vector-valued, and their dispersion relations may depend non-trivially on the values of the potential at every site of the fundamental domain and also on the spectral parameter in a more delicate fashion than in the single-vertex case. In particular, the potential values and spectral parameter $\lambda$ appear in the tangent cone asymptotics in a highly nontrivial fashion, in contrast to \cite{LiuPreprint:Irreducibility} and \cite{FLM1}, where the asymptotics are independent of the potential and of the spectral parameter as well. In fact, this dependence explains one of the crucial technical challenges of this work.

The rest of the paper is laid out as follows. In Section~\ref{Sec:setting}, we describe the general setting of our abstract results. The proofs of those abstract theorems are then given in Section~\ref{Sec:thmproofs}. Since we consider many-vertex models, we need to use the matrix-valued version of Floquet theory, which we describe in Section~\ref{Sec:Floquet}. Finally, we give our main applications to Lieb and decorated lattices in Sections~\ref{Sec:Lieb} and \ref{Sec:decorated}, respectively.

\addtocontents{toc}{\protect\setcounter{tocdepth}{0}}

\section*{Acknowledgements} The authors gratefully acknowledge support from the Simons Center for Geometry and Physics for the Workshop on ``Ergodic operators and quantum graphs'' and ICERM for the Hot Topics workshop on ``Algebraic Geometry in Spectral Theory'', at which some of this work was done.

\addtocontents{toc}{\protect\setcounter{tocdepth}{1}}

\section{Setting and Statement of Main Results}\label{Sec:setting}

 Let $\C^{\star}=\C\backslash\{0\}$,  $\overline{\C} = \C \cup \{\infty\}$, and for $m \in \N$, let $\polym$ (respectively $\laurm$) denote the set of polynomials (respectively Laurent polynomials) in $m$ variables. We denote a typical element of $\laurm$ by
\begin{equation} \label{eq:typicalLaurentDef}
    f(x) = \sum_{\alpha \in \Z^m} c_\alpha x^\alpha, \quad x \in (\C^\star)^m,
\end{equation}
where 
$$x^\alpha = x_1^{\alpha_1} \cdots x_m^{\alpha_m}, \text{ for } x \in (\C^\star)^m, \ \alpha \in \Z^m.$$  
As usual, $f \in \laurm$ is a \emph{unit} if and only if $1/f \in \laurm$, which in turn holds if and only if $f$ is a nonzero monomial: $f(x) = cx^\alpha$ for some $c \neq 0$ and $\alpha \in \bbZ^d$. If $f \not\equiv 0$ is not a unit, it is \emph{irreducible} if it cannot be factored nontrivially. Equivalently, $f$ is irreducible if and only if $f=gh$ implies that one of $g,h$ must be a unit. Given $f \in \laurm$, we let 
\[V(f) = \{x \in (\C^\star)^m: f(x)=0\}\]
denote the variety associated to $f$, and we say that $f$ \emph{meets} $y \in \overline{\C}^m$ if the closure of $V(f)$ contains $y$ (with the closure being taken with respect to the usual product topology on $\overline{\C}^m$).  Define
$$\mathbf{0}_k= (\underbrace{0,0,\ldots,0}_{k \text{ copies}}),$$
and notice that $f$ meets $(\mathbf{0}_{m-1},\infty)$ if and only if $f(\hat{x})$ meets $\mathbf{0}_m$, where $\hat{x} = (x_1,x_2,\ldots,x_{m-1},x_m^{-1})$. We also denote pointwise operations via 
\begin{equation}
    x \odot y = (x_1y_1,\ldots,x_m y_m) \quad \text{ and } \quad x^{\odot \alpha} = (x_1^{\alpha_1}, \ldots,x_m^{\alpha_m}).
\end{equation}

Given a Laurent polynomial as in \eqref{eq:typicalLaurentDef}, set
\begin{equation} \label{eq:calADef}
    \mathcal{A} = \mathcal{A}(f) = \{\alpha \in \bbZ^m : c_\alpha \neq 0\}.
\end{equation} We define $\alpha_{\min}(j)=\alpha_{\min,f}(j)=\min\{\alpha_j\,\,:\,\,\alpha\in \mathcal{A}\}$, that is, $\alpha_{\min}(j)$ is the lowest exponent of $x_j$ in $f(x)$. We also introduce \begin{equation} \label{eq:alpha0jdef}
    \alpha_{0}(j)
    =\alpha_{0,f}(j)
    =\max\{-\alpha_{\min,f}(j),0\}
\end{equation} so that
\begin{equation}\label{eq:minexp}\alpha_0
=\alpha_{0,f}
:= \left(\alpha_{0,f}(1),\ldots,\alpha_{0,f}(m)\right)
\end{equation} 
is the vector in $\bbZ_+^{m}$ with smallest length such that  $f^{+}(x):=x^{\alpha_0}f(x)$ is a polynomial. The interplay between (ir)reducibility properties of $f(x)$ (as a Laurent polynomial) and $f^{+}(x)$ (as a polynomial) plays a key role in this manuscript. In particular, let us note that $f$ is irreducible as a Laurent polynomial if and only if $f^+$ has only one irreducible factor that is not a monomial. This motivates the following definitions.

\begin{definition}\label{def:properpol}
	A \emph{nontrivial monomial} is a Laurent polynomial of the form $cx^\alpha$ where $c \neq 0$ and $\alpha \neq \mathbf{0}_m$. If $cx^\alpha$ is nontrivial and $\alpha_j \geq 0$ for all $j$, we call $cx^\alpha$ a \emph{positive\footnote{Obviously, positive refers to the vector of exponents, not to the range of values taken by the monomial.} monomial}. 
	The degree of $f$ is defined as
	$\deg (f)=\alpha_1+\alpha_2+\cdots +\alpha_m$. 
	Abusing notation slightly, we also denote $\deg(\alpha) = \alpha_1 + \alpha_2 + \cdots +\alpha_m$ for the multi-index $\alpha = (\alpha_1,\ldots,\alpha_m) \in \bbZ^m$.
  We say that $f \in \polym$ is a \emph{proper} polynomial if $f$ has no positive monomial factors. 
 \end{definition}

\begin{definition}\label{lowestdfn}
	Given a Laurent polynomial
	\[f(x)= \sum_{\alpha} c_\alpha x^\alpha,\]
	let $L_- = \min\{\deg(\alpha) : c_\alpha \neq 0\}$. Then,  the \emph{lowest degree component} of $f$ is defined to be the Laurent polynomial 
	\[\underline{f}(x)=\sum_{\deg \alpha =L_{-}}c_\alpha x^{\alpha}.\]
 We sometimes refer to this simply as the \emph{lowest component} of $f$. 

 More generally, if ${l} \in  \bbZ^m \setminus\{0\}$, then the ${l}$-degree of $x^\alpha$ and $\alpha$ are defined by 
 \begin{equation} \deg_{{l}}(x^\alpha) 
 = \deg_{{l}}(\alpha) 
 := \deg(x^{{l}\odot \alpha})
 = \deg({l}\odot \alpha)
 = \sum_{j=1}^m l_j\alpha_j.
 \end{equation}
 One can then put $L_-({{l}}) = \min\{\deg_{{l}}(\alpha) : c_\alpha \neq 0\}$ and the \emph{component of lowest ${l}$-degree} is
 \begin{equation}
     \underline{f}_{{l}}(x) = \sum_{\deg_l(\alpha) = L_-(l)} c_\alpha x^\alpha.
 \end{equation}
\end{definition}

The component of lowest ${l}$-degree has the helpful property that $\underline{(fg)}_{{l}} = \underline{f}_{{l}} \underline{g}_{{l}}$, so (ir)reducibility results for the lowest-degree components can give helpful information about (ir)reducibility of a given polynomial.

One can also view the lowest degree component of a Laurent polynomial in a geometrical way. While our approach does not rely on the following perspective, it adds helpful context. The \emph{Newton polytope} of the Laurent polynomial $f$ is defined to be the convex hull of the set $\mathcal{A}(f)$ given by \eqref{eq:calADef}. Then, any vector $l \in \Z^m \setminus\{0\}$ determines a corresponding face of the polytope via
\[F = F_l = \left\{ y \in \mathcal{N}(f) : \langle y,l \rangle = \min_{y' \in \mathcal{N}(f)} \langle y',l \rangle\right\}. \] The \emph{facial polynomial} corresponding to the face $F$ is
\[f_F(x)= \sum_{\alpha \in F} c_\alpha x^\alpha.\]
 The reader can then check that  $ \underline{f}_{{l}}(x)$ is exactly the Newton polytope facial polynomial associated with the face $F_l$.
\bigskip

With the basic definitions, we are now ready to formulate our main results.
Let $\mathcal{P},\tcp \in \laurm$, and suppose that there exists $l\in \bbN^m$ such that
\begin{equation}\widetilde{\mathcal{P}}(x)={\mathcal{P}}(x^{\odot l}).
\end{equation}
Let $\widetilde{h}_0$ be the lowest degree component of $\tcp^+$, $\widetilde h_{\infty}$ be the lowest degree component of   $(\tcp(\hat x))^+$, where $\hat x=(x_1,x_2,\cdots, x_{m-1}, x_m^{-1})$. Notice that both $\widetilde h_0$ and $\widetilde h_\infty$   are  polynomials.  


  Under these assumptions, there exist $h_0, h_\infty \in \C[x_1,\ldots,x_m]$ such that 
  \[\widetilde{h}_0(x)=h_0(x^{\odot l}), \quad
  \widetilde{h}_{\infty}(x)=h_{\infty}(x^{\odot l})
  \]
  for example, by \cite[Lemma~3.1]{FLM1}.  
  We remark that $h_0$ is the component of lowest $l$-degree of $\mathcal{P}^+$ and
   $h_{\infty}$ is the component of lowest $l$-degree of $(\mathcal{P}(\hat x))^+$.

For the first result of this note, we will make the following assumptions:
\begin{enumerate}[label={\rm(\subscript{A}{{\arabic*}})}]
\item\label{assump1}
Each factor of ${\mathcal{P}}(x)$  meets either $\mathbf{0}_m$ or $(\mathbf{0}_{m-1},\infty)$. 
\medskip

\item\label{assump2} 
Both ${\mathcal{P}}^+(x)$ and  $({\mathcal{P}}(\hat x))^+$ are  proper polynomials, in the sense of Definition~\ref{def:properpol}.
\medskip

\item\label{assump2.5}  $h_0$ has $p_1$ irreducible factors and $h_{\infty}$ has $p_2$ irreducible factors.
\medskip

\item\label{assump3}
$\deg( \widetilde{h}_{0})+ \deg((\widetilde{h}_{\infty} (\hat x))^+)
>\deg(\widetilde{\mathcal{P}}^+)$ or 
$\deg(( \widetilde{h}_{0} (\hat x))^+)+ \deg((\widetilde{h}_{\infty} (x)))>\deg((\widetilde{\mathcal{P}}(\hat x))^+)$

\end{enumerate}
As we shall see in several examples, assumptions \ref{assump1}-\ref{assump3} can be verified in many cases of interest.

\begin{remark}\label{re1}
If  ${\mathcal{P}}^+(x)$ and $({\mathcal{P}}(\hat{x}))^+$  are proper polynomials, then $({\mathcal{P}}(\hat{x}))^+=({\mathcal{P}}^+(\hat{x}))^+.$ 
\end{remark}

\begin{theorem}\label{thm1}
Under assumptions \ref{assump1}-\ref{assump3}, the Laurent polynomial $\mathcal{P}(x)$ has at most $p_1+p_2-1$ irreducible components.
\end{theorem}
Theorem~\ref{thm1} will be proved in Section \ref{Sec:thmproofs}.
 As an immediate consequence we obtain the following corollary.

\begin{coro} If $p_1=p_2=1$, assumptions \ref{assump1}--\ref{assump3} imply that the Laurent polynomial $\mathcal{P}(x)$ is irreducible.
\end{coro}

It the sequel we consider the situation when we have an equality between the degrees, that is, when assumption \ref{assump3} is replaced by

\begin{enumerate}[label=(\subscript{A'}{{\arabic*}})]\addtocounter{enumi}{3}

		\item
  \label{assumpA'3}
	$\deg \widetilde{h}_{0}+ \deg((\widetilde{h}_{\infty} (\hat x))^+)=\deg{\widetilde{\mathcal{P}}}^+$  and
	$\deg(( \widetilde{h}_{0} (\hat x))^+)+ \deg((\widetilde{h}_{\infty} (x)))=\deg(({\widetilde{\mathcal{P}}(\hat x)})^+)$
	\end{enumerate}

	\begin{theorem}\label{thm2}
	Under assumptions \ref{assump1}, \ref{assump2}, \ref{assump2.5}, and  \ref{assumpA'3},  either the Laurent polynomial $\mathcal{P}(x)$ has at most $p_1+p_2-1$ irreducible components or $\mathcal{P}(x)$ has  $p_1+p_2$ irreducible components  and $\widetilde{\mathcal{P}}^{+}(x)=C{\widetilde{h}}_0(x)({\widetilde{h}}_{\infty} (\hat x))^+$ where $C$ is a constant. 
\end{theorem}

As discussed before, Theorems~\ref{thm1} and \ref{thm2} improve the results of \cite{FLM1} in several ways. First, the approach of \cite{FLM1} is necessarily limited to single-vertex models (that is operators on graphs with a free $\bbZ^d$-action that acts transitively on vertices), since that manuscript worked directly with polynomials enjoying a certain expansion in terms of a single fixed Laurent polynomial and some characters. For more general lattices and especially for lattices with many points in the fundamental domain, that expansion is not valid, so it is necessary to deal directly with $\mathcal{P}$ rather than $p$. The abstract result here formulates conditions purely on asymptotics of the polynomial itself which can be checked directly.

\section{Proofs of Theorems \ref{thm1} and \ref{thm2}} \label{Sec:thmproofs}
Before proving our main theorems we will need to collect a few other results.

\subsection{Technical Lemma}
\begin{lemma}\label{lem:componentcount}
	Let  ${\mathcal{P}}$ be a Laurent polynomial that satisfies assumptions  \ref{assump1}, \ref{assump2}, and \ref{assump2.5}. Then  ${\mathcal{P}}$ has at most $p_1+p_2$ irreducible components. Moreover, if ${\mathcal{P}}(x)$ has exactly $p_1+p_2$ irreducible components then $\mathcal{P}^+
	$ enjoys a factorization of the form
	\begin{equation}
	\mathcal{P}^{+}=fg,
	\end{equation}
	where $f \in \polym$ meets $\mathbf{0}_{m}$ and does not meet $(\mathbf{0}_{m-1},\infty)$ and $g \in \polym$ meets $(\mathbf{0}_{m-1},\infty)$ and does not meet $\mathbf{0}_m$. 
\end{lemma}

\begin{proof} Factoring $\mathcal{P}^+$ into irreducibles and applying Assumption~\ref{assump1}, we can write 
	\begin{equation}
	{\mathcal{P}}^{+}=
	\prod_{i=1}^3 \prod_{j=1}^{a_i}f_{ij},
	\end{equation}
	where each $f_{1j}$ is a polynomial that meets $\mathbf{0}_{m}$ and does not meet $(\mathbf{0}_{m-1},\infty)$, each $f_{2j}$ is a polynomial that meets $(\mathbf{0}_{m-1},\infty)$ and does not meet $\mathbf{0}_{m}$, and each $f_{3j}$ is a polynomial that  meets both $\mathbf{0}_{m}$ and  $(\mathbf{0}_{m-1},\infty)$. Note by assumption \ref{assump2} that no $f_{ij}$ can be a monomial so $a_1+a_2+a_3$ is the number of irreducible components of both $\mathcal{P}$ and of $\mathcal{P}^+$.
	
	For each $i\in\{1,3\},j$, let $\widetilde{g}_{ij}$ denote the component of $\widetilde{f}_{ij}=f_{ij}(x^{\odot l})$ of lowest degree, and 
	$g_{ij}$ be such that $\widetilde{g}_{ij}=g_{ij}(x^{\odot l})$.
	Let $t_{ij}$, $i\in\{1,3\}$ denote the number of irreducible factors of $g_{ij}$.
	
	Since $f_{1j}$ and $f_{3j}$ meet $\mathbf{0}_m$, we have that $t_{ij}\geq 1$ for all $1 \le j \le a_i$, $i=1,3$. In particular,
	\begin{equation} \label{eq:compcount:trsums}
	\sum^{a_i}_{j=1}t_{ij}\geq a_i, \quad i=1,3.
	\end{equation}
 We also observe that by definition $f_{2,j}$ does not meet $\mathbf{0}_{m}$, thus $f_{2,j}(\mathbf{0}_{m}) \neq 0$.\par
	Since $\widetilde h_0(x) =C \prod_{i=1,3} \prod_{j=1}^{a_i} \widetilde g_{ij}$ and
	\begin{equation}
	h_0(x) =C \prod_{i=1,3} \prod_{j=1}^{a_i} g_{ij},
	\end{equation}
	these observations and \eqref{eq:compcount:trsums} give
	\begin{equation}\label{eq:p1comp}
	a_1+a_3 \leq \sum_{i=1,3}\sum^{a_i}_{j=1} t_{ij}	=p_1.
	\end{equation}
Similarly, we find that
	\begin{equation}\label{eq:p2comp}
	a_2+a_3 \leq  p_2.
	\end{equation}
	
	Combining \eqref{eq:p1comp} and \eqref{eq:p2comp},  we conclude that
	\begin{equation}\label{eq:totalcomp}
	a_1+a_2+a_3 \leq a_1+a_2+2a_3 \leq p_1+p_2.
	\end{equation}
	Thus, the total number of components, $a_1+a_2+a_3$, is bounded above by $p_1+p_2$, and \eqref{eq:totalcomp} makes it clear that $a_1+a_2+a_3=p_1+p_2$ forces $a_3=0$, as desired.
\end{proof}

\subsection{Proofs of Theorems}

\begin{proof}[Proof of Theorem \ref{thm1}]
For simplicity, we assume the first half of Assumption~\ref{assump3} holds, that is,
\begin{equation} \label{eq:thm1:assmp4a}
    \deg( \widetilde{h}_{0})+ \deg((\widetilde{h}_{\infty} (\hat x))^+)
>\deg(\widetilde{\mathcal{P}}^+).
\end{equation}
From Lemma \ref{lem:componentcount}, we conclude that  ${\mathcal{P}}^+(x)$ has at most $p_1+p_2$ irreducible factors.  Assume for the sake of contradiction that ${\mathcal{P}}^+(x)$ has exactly $p_1+p_2$ irreducible factors. 
By Lemma \ref{lem:componentcount} we can write 
${\mathcal{P}} ^+=f g $, where where $f$ is a polynomial that meets $\mathbf{0}_m$ and does not meet $(\mathbf{0}_{m-1},\infty)$ and $g$ is a polynomial that meets $(\mathbf{0}_{m-1},\infty)$ and does not meet $\mathbf{0}_m$. 
By definition of $\widetilde h_0$ and $\widetilde h_\infty$ and the assumptions on $f$ and $g$, we have that
\begin{equation}\label{g1}
f(x^{\odot l}) = C_1(\widetilde{h}_0(x)+f_1(x))
\end{equation}
where each term in $f_1(x)$ has degree higher than the degree of $\widetilde{h}_0(x)$ and
\begin{equation}\label{g2}
g(x^{\odot l}) = C_2(\widetilde{h}_{\infty}(\hat {x})+g_1(\hat{x}))x_m^{\alpha'}
\end{equation}
 where 
$g_1(\hat{x})$ consists of higher order terms of $\hat x$ and $\alpha'$ is a suitable constant, namely,
$$
\alpha'=\max\{\alpha_{0,\widetilde{h}_{\infty}({x})}(m),\alpha_{0,g_1}(m)\},
$$ 
with $\alpha_{0}$ defined as in \eqref{eq:alpha0jdef}. In simple terms, $\alpha'$ is defined in a way that  $x_m^{-\alpha'}$ is the lowest power of $x_m$ appearing in the sum $\widetilde{h}_{\infty}(\hat{x})+ g_1(\hat{x})$.  Then, letting $\widetilde f(x)=f(x^{\odot l})$
\begin{equation}\label{g4}
\deg{\widetilde{\mathcal{P}}^{+}}\geq \alpha'+\deg(\widetilde{h}_{\infty}(\hat x))+\deg(\widetilde f)\geq\deg((\widetilde{h}_{\infty}(\hat x))^+)+\deg(\widetilde{h}_0),
\end{equation}
which contradicts our assumption \eqref{eq:thm1:assmp4a}.  Therefore,   $\mathcal{P}$ has at most $p_1+p_2-1$ components, as desired.

The case when $\deg(( \widetilde{h}_{0} (\hat x))^+)+ \deg((\widetilde{h}_{\infty} (x)))>\deg((\widetilde{\mathcal{P}}(\hat x))^+)$ is proved similarly.
\end{proof}

\begin{proof}[Proof of Theorem \ref{thm2}]
	 Assume that ${\mathcal{P}}(x)$  has  $p_1+p_2$ irreducible components. We can follow the proof of Theorem \ref{thm1} to reach the chain of inequalities \eqref{g4}.
	 By the first half of assumption \ref{assumpA'3}, namely $\deg \widetilde{h}_{0}+ \deg((\widetilde{h}_{\infty} (\hat x))^+)=\deg{\widetilde{\mathcal{P}}}^+$, all of these must be equations hence by \eqref{g1} and \eqref{g2}, we have that $\alpha'=\alpha_{0,\widetilde{h}_{\infty}}$ and $\deg(\widetilde f)= \deg(\widetilde{h}_0) $. From the latter, we deduce $f_1\equiv 0$. Similarly, using the second half of assumption \ref{assumpA'3} we find that $g_1\equiv0$, finishing the proof.
\end{proof}

\section{Floquet Transform in the Matrix-Valued Case}\label{Sec:Floquet}

Our main interest in the results contained in Theorems~\ref{thm1} and \ref{thm2} comes from applications to Schr\"odinger operators, in particular to estimating the number of components of the dispersion relation and its restriction to constant energies. The former estimates the number of components of the \emph{Bloch variety} and the latter estimates the number of components associated to the \emph{Fermi variety}. Before stating corollaries of our general results, we include some basic terminology and results from Floquet theory in the discrete case following \cite{FLM1}. For a more complete picture we recommend Kuchment's survey \cite{Kuchment2016BAMS} and references therein. 

Since we are interested in periodic operators on graphs, it will be convenient for us to discuss the matrix-valued version of the Floquet transform. This is well-known. We include a summary to fix notation for the reader's convenience. To that end, fix $d\in \bbN$ and $\nu \in \bbN$, which here denotes the order of matrices and later will denote the number of vertex orbits in a given periodic graph. For a set $S$, we denote by $\bbC^{\nu \times S}$ the set of functions from $\{1,2,\ldots,\nu\} \times S$ to $\bbC$.

We then consider the basic Hilbert space 
\begin{align*}
\ell^2(&\bbZ^d,\bbC^\nu) 
\cong \ell^2(\nu \times \bbZ^d)  \\
& = \left\{\varphi\in \bbC^{\nu \times \bbZ^d} : \sum_{j=1}^\nu \sum_{l \in \bbZ^d} |\varphi(j,l)|^2  < \infty \right\}.\end{align*}

Let us now define the operators of interest on this space. Given 
$$\{a_n^{ij} : n \in \bbZ^d, \ 1 \le i,j \le \nu\}$$
such that $a^{ij}$ has compact support for each $1 \le i,j \le \nu$, the associated matrix-valued Toeplitz operator $A:\ell^2(\bbZ^d, \bbC^\nu)\rightarrow \ell^2(\bbZ^d, \bbC^\nu)$ is given by
\begin{equation} \label{eq:matrixToeplitzOpDef}
    [A\psi](i,l) = \sum_{j=1}^\nu\sum_{l' \in \bbZ^d} a^{ij}_{l-l'}\psi(j,l'), \quad 1\le i \le \nu, \ l \in \bbZ^d. 
\end{equation}
It is helpful to rewrite this in matrix form, especially since we will be describing some other block decompositions later. More specifically, set $\psi(l) = [\psi(1, l),\ldots,\psi(\nu, l)]^\top \in \bbC^\nu$ for $\psi \in \ell^2(\bbZ^d,\bbC^\nu)$  and write 
\[
A(l) = \begin{bmatrix}
    a^{11}_l & \cdots & a^{1\nu}_l \\
    \vdots & \ddots & \vdots \\
    a^{\nu 1}_l & \cdots & a^{\nu \nu}_l
\end{bmatrix} \]
and observe that we can rewrite \eqref{eq:matrixToeplitzOpDef} as 
\begin{equation} 
    [A\psi](l) = \sum_{l' \in \bbZ^d} A(l-l')\psi(l'), \quad  l \in \bbZ^d. 
\end{equation}

To add a potential, we consider $V:\nu \times \bbZ^d \to \bbC $ bounded and define $[V\psi](i,l) = V(i,l)\psi(i,l)$ for $1 \le i \le \nu$ and $l \in \bbZ^d$. We are then interested in the  operator
\begin{equation}
    H = A+V.
\end{equation}

For $q \in \bbN^d$, we say that $V$ is $q$\emph{-periodic} if\footnote{recall that $x\odot y = (x_1y_1,x_2y_2,\ldots,x_dy_d)$} $V(j,n+l\odot q) = V(j,n)$ for all $1\le j \le \nu$ and $n,l \in \bbZ^d$.
In this case, denote by
$$W=\{ w \in \bbZ^d : 0 \le w_i < q_i \ \forall 1 \le i \le d\}$$ 
its fundamental domain and write
$\Gamma=\{l\odot q: l\in \bbZ^d\}$ for the period lattice. Writing $q^* = (q_1^{-1},\ldots,q_d^{-1})$, the dual lattice is given by $\Gamma^{\ast}=\{l \odot q^* : l \in \bbZ^d\}= \{(\tfrac{l_1}{q_1},\ldots,\frac{l_d}{q_d}):\,\,l\in \bbZ^d\}.$
Let us now define $\bbT^d_* = \bbR^d/\Gamma^*$,
\begin{align*}
    \scrH_{q,\nu} 
    &= \left\{ f: \nu  \times W \times \bbT^d_* \to \bbC : \sum_{j=1}^\nu \sum_{w \in W} \int |f(j,w,k)|^2 \frac{dk}{|\bbT^d_*|} < \infty \right\} \\
    & \cong \bigoplus_{j=1}^\nu \int_{\bbT^d_*}^\oplus \bbC^W \, \frac{dk}{|\bbT^d_*|} \\
    & \cong L^2\left(\bbT^d_*, \bbC^{\nu \times W}; \frac{dk}{|\bbT^d_*|} \right).
\end{align*} 
The first version of the Floquet transform is given by $\scrF_{q,\nu} : \ell^2(\bbZ^d,\bbC^\nu) \to \scrH_{q,\nu}$, $u \mapsto \widehat{u}$ where
\[
\widehat{u}(j,w,k) 
= \sum_{n \in \bbZ^d} e^{-2 \pi \iop \langle n \odot q,k\rangle}u(j,w+n\odot q), 
\]
for $1 \le j \le \nu$, $w \in W$, and $k \in \bbT^d_*$. This conjugates $H$ to a decomposable operator whose action on the $k$th fiber of $\scrH_{q,\nu}$ is given by the restriction of $H$ to $\{1,\ldots,\nu\} \times W$ with boundary conditions 
\begin{equation}\label{bc}
u(j, n+m\odot q)  = e^{2\pi \iop \langle m\odot q,k\rangle } u(j,n), \  n,m\in\Z^d, \ 1\le j \le \nu.
\end{equation}
More precisely, for $k \in \bbT^d_*$, define $\widetilde{H}(k)$ on $\bbC^{\nu \times W}$ by 
\begin{align}\nonumber
    [\widetilde{H}&(k)u](i,w) \\
    \label{eq:tildeHkdef}
    & =V(i,w)u(i,w) 
    + \sum_{j=1}^\nu \sum_{w' \in W} \sum_{l \in \bbZ^d} e^{2\pi \iop \langle l \odot q, k \rangle} a^{ij}_{w-(w'+l\odot q)}u(j,w') 
\end{align}
for $1\le i \le \nu$, $w \in W$. We have the following:

\begin{prop}\label{prop:floquetform1}
The operator $\scrF_{q,\nu}$ is unitary. If $V$ is $q$-periodic, then 
\[
\scrF_{q,\nu} H \scrF_{q,\nu}^* = \int^\oplus_{\bbT^d_*} \widetilde{H}(k) \, \frac{dk}{|\bbT^d_*|},
\]
where $\widetilde{H}(k)$ is as in \eqref{eq:tildeHkdef}.
\end{prop}
\begin{proof} 
This follows from a direct calculation.
\end{proof}

For our purposes, the form of $H$ from Proposition~\ref{prop:floquetform1} is not optimal, since the dependence of $\widetilde{H}(k)$ on the Floquet multipliers can be highly nontrivial. In order to compute asymptotic terms, it is useful to have a different version of the Floquet transform that places the multipliers on a (block) diagonal. 

To describe this, let $Q = \#W = q_1q_2\cdots q_d$, and let us consider the $\nu Q$-dimensional vector space $\bbC^{\nu \times W}$ which consists of vectors $u:\{1,2,\ldots,\nu\} \times W \to \bbC$. For $k \in \bbR^d$, the corresponding fiber of the \emph{Floquet transform} $\scrF(k):\bbC^{\nu \times W} \to \bbC^{\nu \times W}$ is given by
\begin{equation}
    [\scrF(k)u](j,w) = \frac{1}{\sqrt{Q}} \sum_{n \in W} e^{-2\pi \iop  \langle w \odot q^* + k, n \rangle}u(j,n).
\end{equation}

It turns out that $\scrF(k) \widetilde H(k) \scrF(k)^*$ takes a simple form: the sum of a block-diagonal operator that depends only on $A$ and $k$ and a block-operator that depends only on $V$ (and not on $k$).

 For $1 \le i,j \le \nu$, define
\begin{equation}
    p_{ij}(z)=\sum_{n\in \bbZ^d} a^{ij}_{n}z^{-n}
\end{equation}
and consider $p(z) \in \bbC^{\nu \times \nu}$ given by
\begin{equation} \label{eq:matrixToeppzdef}
p(z)=\begin{bmatrix} p_{11}(z)&p_{12}(z)&\ldots&p_{1\nu }(z)\\
 p_{21}(z)&p_{22}(z)&\ldots&p_{2\nu }(z)\\
 \vdots & \vdots & \ddots & \vdots \\
  p_{\nu 1}(z)&p_{\nu 2}(z)&\ldots&p_{\nu \nu }(z).
\end{bmatrix}
\end{equation}
Given  $q \in \bbN^d$ and $n \in \bbZ^d$, the vector 
$\mu_n=\mu_{n,q} \in (\bbC^\star)^d$ is defined by
\begin{equation}\label{action}\mu_n^j = e^{2\pi \iop n_j/q_j}, \quad 1 \le j \le d.
\end{equation}

Using $\mu_n$ as in \eqref{action}, define
\begin{equation}\label{eq:diag:Blockmatrix}
\mathcal{D}_z(n,n^\prime) = p(\mu_{n}\odot z)\delta_{n,n^{\prime}},\,\,n,n^\prime \in W, 
\end{equation}
where $z=(e^{2\pi \iop k_1}, \ldots, e^{2\pi \iop k_d})$, which we abbreviate as $z = \exp(2\pi \iop k)$.

The linear operator $D_z$ on $\bbC^{\nu \times W}$ is then the (block-diagonal) operator $D_z = \diag(\mathcal{D}_z(n,n))$, that is,
\begin{equation} \label{eq:DZdef}
    [D_z u](w) = p(\mu_w \odot z) u(w), \quad w \in W,
\end{equation}
where similar to before, we denote $u(w) = [u(1,w),\ldots,u(\nu,w)]^\top \in \bbC^\nu$.
Equivalently, choosing an enumeration $W = \{w_1,\ldots,w_Q\}$,  one can write $D_z$ as a block-diagonal matrix
\begin{equation}
    D_z = \begin{bmatrix}
        p(\mu_{w_1}\odot z) \\
        & p(\mu_{w_2}\odot z) \\
        && \ddots \\
        &&& p(\mu_{w_Q}\odot z)
    \end{bmatrix}
\end{equation}
with $\nu \times \nu$ blocks on the diagonal and in which all unspecified blocks are zero.

Writing the discrete Fourier transform of a $q$-periodic function $g:\bbZ^d\rightarrow \bbC$ on $\Gamma^*$ by
$$ \widehat{g}_l=\frac{1}{\sqrt{Q}}\sum_{n\in W}e^{-2\pi \iop \langle l,n\rangle }g_n,\quad l\in \Gamma^*,$$
 we define $\mathcal{B}=\mathcal{B}_V$ by
\begin{equation}\label{eq:offdiag:Blockmatrix}
\mathcal{B}(n,n')=\diag
\left( \widehat V_1\left( q^* \odot(n-n')\right),
\ldots,
\widehat V_\nu \left(q^* \odot(n-n')\right)\right)
,\,\,n,n^\prime \in W,
\end{equation}
where we recall $q^* = (q_1^{-1},\ldots, q_d^{-1})$.

The corresponding operator $B=B_V$ on $\bbC^{\nu \times W}$ is given in block matrix form in a similar fashion to $D_z$:
\begin{equation}
    [B_V u](w) = \sum_{w' \in W} \mathcal{B}_V(w,w')u(w').
\end{equation}
Enumerating $W = \{w_1,\ldots,w_Q\}$ as before, one can write $B_V$ as the $Q \times Q$ block matrix
\begin{equation} \label{eq:BVdef}
    B_V = \begin{bmatrix}
        \mathcal{B}(w_1,w_1) & \mathcal{B}(w_1,w_2) & \cdots & \mathcal{B}(w_1,w_Q) \\
        \mathcal{B}(w_2,w_1) & \mathcal{B}(w_2,w_2) & \cdots & \mathcal{B}(w_2,w_Q) \\
        \vdots & \vdots & \ddots & \vdots \\
        \mathcal{B}(w_Q,w_1) & \mathcal{B}(w_Q,w_2) & \cdots & \mathcal{B}(w_Q,w_Q) \\
    \end{bmatrix}
\end{equation}
where each block is $\nu \times \nu$.

\begin{prop}\label{prop:floquetTransf}
Assume $V$ is $q$-periodic. 
Then 
\begin{equation}
    \scrF(k) \widetilde H(k) \scrF(k)^* = D_z+B_V
\end{equation}
for each $k$, where $D_z$ and $B_V$ are as above. In particular, writing $\scrF = \int^\oplus_{\bbT^d_*} \scrF(k)$, one has
\begin{equation}
    \scrF \widetilde{H} \scrF^* = \int^\oplus_{\bbT^d_*} (D_{\exp 2\pi \iop k} + B_V) \, \frac{dk}{|\bbT^d_*|}.
\end{equation}
\end{prop}

\begin{proof}
    This follows from a calculation almost identical to that in the proof of \cite[Proposition 5.3] {FLM1}.
\end{proof}

\begin{prop} \label{prop:tcpCharInv}
    Assume $V$ is $q$-periodic, and let
    \begin{equation}
        \widetilde{\mathcal{P}}(z,\lambda) = \det(D_z-B_V-\lambda I),
    \end{equation}
    where $D_z$ and $B_V$ are as before. For any $w \in W$, 
    \begin{equation}
        \widetilde{\mathcal{P}}(\mu_w \odot z, \lambda) \equiv \widetilde{\mathcal{P}}(z,\lambda).
    \end{equation}
    In particular, $\widetilde{\mathcal{P}}(z,\lambda) = \mathcal{P}(z^{\odot q},\lambda)$ for a suitable Laurent polynomial $\mathcal{P}$.
\end{prop}

\begin{proof}
    For $w \in W$, let $T_w:\bbC^{\nu \times W} \to \bbC^{\nu \times W}$ denote the operator
    \[[T_w u](j,w') = u(j,\{w'-w\}), \quad 1 \le j \le \nu, \ w' \in W,\]
where for $n \in \bbZ^d$, $\{n\}$ denotes the unique element of $W$ that is equivalent to $n$ modulo $\Gamma$. By direct calculations, $T_w^* D_z T_w = D_{\mu_w \odot z}$ and $T_w^* B_V T_w = B_V$, so
\begin{align*}
\widetilde{\mathcal{P}}(\mu_w \odot z,\lambda)
 = \det(T_w^* (D_z  + B_V -\lambda)T_w) 
 = \widetilde{\mathcal{P}}(z,\lambda),
\end{align*}
as desired. The second statement then follows immediately from \cite[Lemma~3.1]{FLM1}.
\end{proof}

Let us conclude this section by pointing out that the framework described above includes all finite-range translation-invariant operators on periodic graphs in a canonical fashion.

More precisely, recall that a \emph{graph}  $\calG$  consists of a nonempty set $\calV$ of vertices and $\calE$  a set of unordered pairs of elements of $\calV$, called edges. A \emph{$\bbZ^d$-periodic graph} is a locally finite graph $\calG$ on which $\bbZ^d$ acts freely and cocompactly. We denote the action of $\bbZ^d$ additively, so the assumption that $\bbZ^d$ acts freely can be expressed as
\[v+{n} \neq v, \quad \forall {n} \neq \mathbf{0}_d.\]
The translations $v \mapsto v+{n}$ induce unitary operators $U_{ n}:\ell^2(\calV) \to \ell^2(\calV)$ via $[U_{ n}f](v)=f(v+{n})$. One says that a bounded operator $A$ on $\ell^2(\calV)$ is \emph{translation-invariant} if $AU_{ n} = U_{ n}A$ for all ${n} \in \bbZ^d$.

Consider a translation-invariant operator $A$
and choose a fundamental domain $\mathcal{V}_f$ so that $\mathcal{V} = \bigcup_{ n \in \bbZ^d}(\mathcal{V}_f + {n})$. Let $\nu = \#\calV_f$, and write $\calV_f = \{v_1,\ldots,v_\nu\}$. Thus, any $v \in \calV$ can be written as $v_j+{n}$ for some $j \in \{1,2,\ldots, \nu\}$ and ${n} \in \bbZ^d$. 

We further assume that $A$ has finite range in the sense that there exists $r>0$ such that
\begin{equation}
    \langle \delta_{v_i+n}, A\delta_{v_j+m} \rangle =0 \quad \text{whenever } |n-m|> r.
\end{equation}
The map $\Phi:\ell^2(\calV) \to \ell^2(\bbZ^d,\bbC^\nu)$ given by
\[[\Phi f](j,n) =  f(v_j+{n})\]
is unitary, and $\Phi A \Phi^{*}$ is the matrix Toeplitz operator corresponding to coefficients $a_n^{ij}$ given by
\begin{equation}
    a^{ij}_n = \langle \delta_{v_i+n}, A\delta_{v_j} \rangle,
\end{equation}
 as above.

\section{Applications to the Lieb Lattice}\label{Sec:Lieb}

 The Lieb lattice in dimension $d=2$ is the graph with vertices $\calV$ consisting of all $n \in \bbZ^2$ such that $n_1$ and $n_2$ are not both odd. One then connects a pair of vertices $n$ and $m$ by an edge whenever $\|n-m\|_1=  1$. Compare Figure~\ref{fig:Liebd2}; there, the black vertices are those that belong to $2\bbZ \oplus 2 \bbZ$.
 
 We let $A$ denote the adjacency operator on the Lieb lattice and the corresponding operator $A+V$ acting on $\ell^2(\bbZ^2, \bbC^{3})$, where $V: \bbZ^2 \to \bbC^3$ is $q$-periodic. Let us exploit the connection discussed in the previous section to write this as a matrix-valued periodic operator. Referring to Figure~\ref{fig:Liebd2}, let $\psi_1$, $\psi_2$, and $\psi_3$ respectively represent the value of $\psi \in \ell^2(\bbZ^2, \bbC^{3})$ at the corresponding black, red, or blue vertex respectively. Writing $\{e_1,e_2\} \subseteq \bbZ^2$ for the standard basis and defining $\Delta_{\pm e_j}$ on $\ell^2(\bbZ^2)$ by 
 \[
 [\Delta_{\pm e_j}\psi](n) =\psi(n) +  \psi(n \pm e_j), 
 \quad \psi \in \ell^2(\bbZ^2),\]
the reader can theck that
 \begin{equation}\label{eq:AdjLieb}
 [A\psi](n)=\begin{bmatrix} [\Delta_{-e_1}\psi_2](n) + [\Delta_{-e_2}\psi_3](n)\\
[\Delta_{e_1}\psi_1](n)\\
[\Delta_{e_2}\psi_1](n)
\end{bmatrix}.
\end{equation}
In matrix form:
\begin{equation} \label{eq:LiebasMatrixToeo}
     A = \begin{bmatrix}
         0 & \Delta_{-e_1} &  \Delta_{-e_2} \\
         \Delta_{e_1} & 0 & 0  \\
         \Delta_{e_2} & 0 & 0
     \end{bmatrix}.
 \end{equation}

 Let us now use Proposition~\ref{prop:floquetTransf} to write the Floquet matrix of $A+V$, with $A$ as in \eqref{eq:AdjLieb}, as the $3Q\times 3Q$ matrix
$D_z + B_V$.

By \eqref{eq:matrixToeppzdef} and \eqref{eq:LiebasMatrixToeo}, $p(z)$ is given by
\begin{equation}\label{eq:Lieb}
    p(z) =\begin{bmatrix}0&1+z^{-1}_1&1+z^{-1}_2\\
    1+z_1 & 0&0\\
    1+z_2 & 0&0
    \end{bmatrix}
\end{equation}
so \begin{equation}
    \det (p(z) -\lambda I)=(-\lambda)(z^{-1}_1+z_1+z^{-1}_2+z_2+\lambda^2+4)
\end{equation}
Notice that the factor of $\lambda$ shows the existence of a flat band of the operator $A$. Indeed, one can check that
\begin{equation}
    \psi(n) = \begin{cases}
        +1 & n = (0,1),\ (2,1)\\
        -1 & n= (1,0),\ (1,2)\\
        0 & \text{otherwise}
    \end{cases}
\end{equation}
defines an eigenfunction of $A$  with eignevalue $0$ having compact support.

Recall that the matrix $B = B_V$ is a $3Q \times 3Q$ matrix given by \eqref{eq:offdiag:Blockmatrix} and \eqref{eq:BVdef} with blocks of the form
\begin{equation}\label{eq:offdiagLiebnd}
\mathcal{B}(n,n')=
\begin{bmatrix}
\widehat{V_1}(n-n')&0&0\\
0&\widehat{V_2}(n-n')&0\\
0&0&{\widehat V}_{3}(n-n')
\end{bmatrix}
\quad n, n' \in W.
\end{equation}

In comparison to the adjacency operator of $\bbZ^2$,  the inclusion of extra vertices accounts for additional interactions, which increases the complexity of the study of the Fermi variety in a substantial fashion, especially in the presence of a periodic potential. For nontrivial choices of $V$, counting the number of irreducible components of $\widetilde{\mathcal{P}}(z)=\det (D_z+B_V-\lambda I)$ and the related $\mathcal{P}(z)$ is significantly harder due to the contributions from the off-diagonal blocks. We shall explain below how irreducibility of the latter follows from Theorem \ref{thm2} for all but finitely many values of $\lambda$.
 
Before proving the main irreducibility result for the Lieb lattice, we will need a suitable technical result about the asymptotic terms. Let us fix some notation. Fix $\lambda \in \bbC$. Let $\widetilde {\mathcal{P}}(z)  = \det(D_z + B_V - \lambda I)$ with ${D}_z$ and $B_V$ given by \eqref{eq:DZdef} and \eqref{eq:BVdef}. By Proposition~\ref{prop:tcpCharInv}, there is a polynomial $\mathcal{P}$ such that $\widetilde{\mathcal{P}}(z) = \mathcal{P}(z^{\odot q})$. Define $\widetilde{h}_0$ $h_0$, $\widetilde{h}_{\infty}$ and $h_{\infty}$ as in Section~\ref{Sec:setting}. 

Since we are interested in the \emph{Fermi} variety, we want to consider $\mathcal{P}$ as a polynomial in the variable $z$ with $\lambda$ fixed. Thus, we generally write $\mathcal{P}(z)$ unless we need to explicitly mention the dependence on $\lambda$, in which case we write $\mathcal{P}(z;\lambda)$, with similar notation for $h_0$, $h_\infty$, and so on.

\begin{lemma}\label{Lem:compcountLieb} 
If $\gcd(q_1,q_2)=1$, then $h_0$ and $h_{\infty}$ are irreducible for all but finitely many $\lambda$.
\end{lemma}

\begin{proof}
We will only prove irreducibility of $h_0$ since the result for $h_{\infty}$ is analogous. From the structure of $D_z$ and $B_V$ one readily checks that the lowest degree term of $\widetilde{\mathcal{P}}(z) = \det(D_z-B_V -\lambda I)$ is of the form
$$r_0(z)=\sum c_{\alpha }z^\alpha$$
where the summation runs over those $ \alpha \in [-q_1q_2,0]^2\cap \bbZ^2$ satisfying $\alpha_1+\alpha_2=-Q = -q_1q_2$.
Moreover, since $r_0(z)=r_0(\mu_n\odot z)$ for each $n\in W$, the only terms in $r_0(z)$ that can have $c_\alpha \neq 0$ are the terms with $\alpha_1=m_1q_1$ and $\alpha_2=m_2q_2$ for integers $-q_2 \leq  m_1\leq 0$ and $-q_1 \leq m_2\leq 0$. Since $\gcd(q_1,q_2)=1$ the only solutions to 
$$m_1q_1+m_2q_2=-q_1q_2$$
within the given range are $(m_1,m_2)=(-q_2,0)$ and $(m_1,m_2)=(0,-q_1)$. In particular,
$$
r_0(z)=c_2z^{-Q}_1+c_1z^{-Q}_2,$$
where $c_j = c_j(\lambda)$ for $j=1,2.$
From this, it follows that
$$
\widetilde h_0(z)=(z_1z_2)^Qr_0(z)=c_1z^{Q}_1+c_2z^{Q}_2.
$$
We shall see below that $c_1(\lambda)$ and $c_2(\lambda)$ are polynomials in $\lambda$ that do not vanish identically. Consequently, $h_0(z)=c_1z^{q_2}_1+c_2z^{q_1}_2$  is irreducible for all but finitely many $\lambda$; indeed, it is irreducible precisely for those $\lambda $ for which $c_1(\lambda)$ and $c_2(\lambda)$ are both nonzero.
To finish the proof we now show $c_2(\lambda) \not\equiv 0$ (the argument for $c_1(\lambda)$ is analogous) by expanding $\det(D_z + B_V - \lambda I)$ using permutations $\sigma:\{1,\ldots,3Q\}\rightarrow \{1,\ldots,3Q\}$. Note that the highest degree term in $c_2(\lambda)$ is achieved precisely by the permutation $\tau$ given by
\begin{equation}\tau(j)=
    \begin{cases}
     j+1 & \text{if } j\equiv 1\ \mathrm{mod} \ {3}\\
     j-1 & \text{if } j\equiv 2 \ \mathrm{mod} \ {3}\\
     j & \text{otherwise}
    \end{cases}\,,
\end{equation}
which yields a term in $\det(D_z + B_V - \lambda I)$ of the form  $(-1)^Q(\widehat{V_3}(0) - \lambda)^{Q}z_1^{-Q}$. 
It follows that $c_2(\lambda)$ is a polynomial of degree $Q$ in $\lambda$ which completes the proof.
\end{proof}

\begin{lemma}\label{Lem:meetsLieb} Let $\widetilde{\mathcal{P}}(z) = \det(D_z + B_V-\lambda I)$ as above with $\mathrm{gcd}(q_1,q_2)=1$. Let $\mathcal{P}(z)$ satisfy $\widetilde {\mathcal{P}}(z)=\mathcal{P}(z^{\odot q})$ for all $z\in \mathbb{C}$. Then, for all but finitely many values of $\lambda$, each factor of $\mathcal{P}(z)$ meets either $(0,0)$ or $(0,\infty)$.

\end{lemma}
\begin{proof}
Assume $z(n) = (z_1(n),z_2(n) ) \in V(\widetilde{\mathcal{P}})$ with $z(n) \to y=(y_1,y_2) \in \overline{\bbC}^2$ and $y_1=0$. For all but finitely many $\lambda$, this forces $y_2 \in \{0,\infty\}$. Indeed, if $y_2 \in \bbC^\star$, then we follow the  the proof of Lemma \ref{Lem:compcountLieb} to conclude that 
\[\widetilde{\mathcal{P}}(z(n))
=c_2(\lambda)\frac{1}{z_1(n)^Q}+O(z_1(n)^{-(Q-1)}).\]
Whenever $c_2(\lambda)\neq 0$, this gives $\widetilde{\mathcal{P}}(z(n))\to \infty$, which is inconsistent with $z(n) \in V(\widetilde{\mathcal{P}})$.

Since $z(n)\to (0,0)$ (respectively, $z(n) \to (0,\infty)$) if and only if $(z(n))^{\odot q}\to (0,0)$ (respectively, $(z(n))^{\odot q}\to (0,\infty)$) it readily follows that each factor of $\mathcal{P}(z)$  meets either $(0,0)$ or  $(0,\infty)$.
\end{proof}

\begin{coro}\label{thm:Liebd}
If $\gcd(q_1,q_2)=1$ the Fermi variety of the Schr\"odinger operator $A+V$  on the Lieb lattice given by \eqref{eq:AdjLieb} is irreducible for all but finitely many $\lambda \in \mathbb{C}$.
\end{coro}
\begin{proof}
In view of the discussion in previous proofs, when the periods $q_1$ and $q_2$ are coprime $(z_1z_2)^Q\det(D_z + B_V - \lambda I)$ has a lowest degree term  of the form
\begin{equation} \label{eq:lieb:tildeh0decomp}
\widetilde h_0(z)=c_{01}(\lambda,V)z_1^{Q}+c_{02}(\lambda,V)z_2^{Q}
\end{equation}
where each $c_{0j}(\lambda,V)$ is a polynomial of degree $Q$ in $\lambda$ that in principle also depends on  $V(l)$ for all $l\in W$.
 In particular, there exists a finite set $\mathcal{Z}_0$ such that for $\lambda\in \bbC\setminus \mathcal{Z}_0$ $\widetilde{h}_0(z,\lambda)$ does not have any monomial factors. Similarly (or by symmetry considerations) one checks that $\widetilde h_\infty(z) = c_{11}(\lambda,V)z_1^{Q} + c_{12}(\lambda,V)z_2^Q$ where each $c_{1j}(\lambda,V)$ is a polynomial of degree $Q$ in $\lambda$ and therefore that $(\widetilde{h}_{\infty}(\hat z))^{+}=c_{12}(\lambda)+c_{11}(\lambda)z_1^{Q}z_2^{Q}$.
 It follows that $(\widetilde{h}_{\infty}(\hat z))^{+}$ also
  does not have any monomial factors for $\lambda\in \bbC\setminus \mathcal{Z}_{0}$.
 In particular
$$
\deg \widetilde{h}_{0}(z)+\deg(\widetilde{h}_{\infty}(\hat z))^+=\deg \widetilde{\mathcal P}^{+}(z)=3Q.
$$
Therefore, for $\lambda \in \bbC\setminus\mathcal{Z}_0$ the assumptions of Theorem \ref{thm2} are met. We conclude that if ${\mathcal P}(z)$ is reducible then we must have
\begin{align}\label{eq:contradictfactor}(z_1 z_2)^{Q}\widetilde{\mathcal P}(z)&=K[c_{01}(\lambda)z_2^{Q}+c_{02}(\lambda)z_1^{Q}]\\&\times [c_{12}(\lambda)+c_{11}(\lambda)z^Q_1z^Q_2].
\end{align} 
where $K=K(\lambda)$ is a rational function of $\lambda$.
 We will show that the above setting leads to a contradiction for all but finitely many $\lambda$. Indeed, the highest power of $\lambda$ on the left-hand side of \eqref{eq:contradictfactor} equals $3Q$. Since  each of the polynomials $c_1(\lambda), c_2(\lambda)$ has degree $Q$, we conclude from this that $K(\lambda)\sim \lambda^Q$  in the sense that $$\lim_{\lambda \to \infty}\frac{K(\lambda)}{\lambda^Q}=C\neq 0.$$ 
 However, this implies that the coefficient of $z^Q_2$ in the right-hand side is of the order of $\lambda^{3Q}$ 
 whereas the coefficient of $z^Q_2$ on the left-hand side is of the order of $\lambda^Q$, according to the proof of Lemma \ref{Lem:compcountLieb}. Thus, this leads to a contradiction unless $\lambda$ belongs to a certain algebraic set. In particular, $\mathcal{P}(z)$ is irreducible for all but finitely many values of $\lambda$.\end{proof}

\section{Applications to Decorated Lattices}\label{Sec:decorated}

Another application of the main results of this note concerns certain graph decorations of $\bbZ^d$ by polygons of $\nu \in \bbN$ vertices.
Let $\calG = (\calV,\calE)$ denote the graph obtained from $\bbZ^d$ by attaching a cyclic graph of order $\nu$ to each $n \in \bbZ^d$. More precisely, the vertices of $\calG$ are given by $\calV = \bbZ \times \bbZ_\nu$ and edges described as follows: 
\begin{itemize}
    \item For any $n,n' \in \bbZ^d$ with $\|n-n'\|_1 = 1$, there is an edge from $(n,0)$ to $(n',0)$
    \item For any $n \in \bbZ^d$ and $j \in \bbZ_\nu$, there is an edge from $(n,j)$ to $(n,j+1)$ and to $(n,j-1)$.
\end{itemize}
Figure~\ref{fig:triangdecor} illustrates this in the case $d=2$ and $\nu = 3$. The black vertices represent $(n,0)$ with $n \in \bbZ^2$ while the blue vertices represent $(n,\pm 1)$. The case when $\nu=2$, where the graph is decorated by lines, is one of the examples of graphs studied by Aizenman and Schenker to create spectral gaps, see \cite[Figure 2]{AS}.

We now let $A$ denote the adjacency operator on $\ell^2(\mathcal{V})$. As in the previous section, one can write this a periodic operator with matrix coefficients. For $\psi \in \ell^2(\bbZ^d,\bbC^\nu)$ we identify the coordinate function $\psi(j,n)$ with the value of $\psi$ on $(n,j-1 \ \mathrm{mod} \ \nu)$. With these definitions, one has that
\begin{equation}\label{eq:decoradj}
[A\psi](n)
=\begin{bmatrix} 
[\Delta \psi_1](n)+\psi_{2}(n)+\psi_{\nu}(n)\\
\psi_{1}(n)+\psi_3(n)\\
\vdots\\
\psi_{\nu-2}(n)+\psi_{\nu}(n)\\
\psi_{\nu-1}(n)+\psi_1(n)
\end{bmatrix},
\end{equation}
where $\Delta$ in the top row represents the adjacency operator on $\ell^2\left(\bbZ^d\right)$.
A portion of the example where $d=2$ and $\nu=3$ is presented in Figure~\ref{fig:triangdecor}.

This time, the diagonal blocks of the Floquet matrix are generated by
\begin{equation}\label{eq:FloquetmatrixDecor}
p(z)=
\begin{bmatrix}
b_0(z) &1& &&& 1\\
1&0&1&&&\\
&1&0&1&&\\
&&\ddots&\ddots&\ddots&\ \\
 &&&1&0&1 \\
1&&&&1&0
\end{bmatrix},
\end{equation}
where $b_0(z)=z_1+z^{-1}_1+\ldots+z_d+z^{-1}_d$ and all unspecified entries are $0$.
We now consider the operator $H = A+V$ with a $q$-periodic potential $V:\bbZ^d \to \bbC$.

We will check below that that the dispersion relation
$\widetilde{\mathcal{P}}(z) = \det(D_z + B_V -\lambda I)$ of $A+V$, with $A$ as in \eqref{eq:decoradj} and $V$ $q$-periodic, satisfies the assumptions of Theorem~\ref{thm1} (respectively \ref{thm2}) whenever $d>2$ (respectively $d=2$) for all but finitely many $\lambda \in \bbC$.  Define $\mathcal{P}$, $\widetilde{h}_0$, $h_0$,   $\widetilde{h}_{\infty}$, and $h_{\infty}$ as in Section \ref{Sec:setting}.

Note that, under the above definitions
\begin{equation}\label{eq:h0zd}
\widetilde h_0(z;\lambda)=\widetilde h_{\infty}(z;\lambda)=s(\lambda)(z_1\cdots z_d)^Q\prod_{n \in W} r_0(\mu_n\odot z)
\end{equation}
with $r_0(z)=z^{-1}_1+\ldots+z^{-1}_d$ and $s(\lambda)$ a polynomial in $\lambda$ of degree $(\nu-1)Q$.

Consequently,
\begin{equation}\label{eq:hinftyzd}
(\widetilde h_{\infty}(\hat z;\lambda))^{+}=s(\lambda)(z_1\cdots z_{d-1})^Q\prod_{n \in W} r_{\infty}(\mu_n\odot z)
\end{equation}
with $r_{\infty}(z)  = r_0(\hat z) =z^{-1}_1+\ldots+z^{-1}_{d-1}+z_d$ and $s(\lambda)$ is as above. In particular, whenever $s(\lambda)\neq 0$ we have that $\widetilde h_0(z)$ and $\widetilde h_{\infty}(z)$ are nonzero.

Recalling that
\begin{equation}\label{eq:h0hinfty}
    \widetilde h_0(z)=h_0(z^{\odot q})\,\, \text{and}\,\, \widetilde h_{\infty}(z)= h_{\infty}(z^{\odot q}),
\end{equation}
 we have the following.
\begin{lemma}\label{lem:lowestcompdecor}
 If $\mathrm{gcd}(q_1,\ldots,q_d)=1$, then $h_0$ and $h_{\infty}$ given by \eqref{eq:h0zd}, \eqref{eq:hinftyzd}, \eqref{eq:h0hinfty} are irreducible for all but finitely many $\lambda$. The finite set of exceptional $\lambda$'s is precisely the set on which $s(\lambda)=0$.    
\end{lemma}
\begin{proof}
See~\cite[Lemma 5.1]{LiuPreprint:Irreducibility}.
\end{proof}
\begin{lemma} Let $\widetilde{\mathcal{P}}(z) = \det(D_z + B_V- \lambda I)$, and let $\mathcal{P}(z)$ satisfy $\widetilde {\mathcal{P}}(z)=\mathcal{P}(z^{\odot q})$ for all $z\in \mathbb{C}$. Then, for all but finitely many values of $\lambda$, each factor of $\mathcal{P}(z)$  meets either $\mathbf{0}_d$ or $(\mathbf{0}_{d-1},\infty)$. 

\end{lemma}
\begin{proof}
If $\widetilde{\mathcal{P}}(z)=0$ with $(z_1,\ldots,z_{d-1})\to  \mathbf{0}_{d-1}$ then we must have that $z_d\to 0$ or $z_d\to \infty$ since otherwise we would have $\widetilde{\mathcal{P}}(z)=s(\lambda)(z^{-Q}_1+\ldots+z^{-Q}_{d-1})+\text{l.o.t}$ with $s(\lambda)$ the polynomial from \eqref{eq:h0zd}. In particular whenever $s(\lambda)\neq 0$ this would imply that $\widetilde{\mathcal{P}}(z)\to \infty$, which is a contradiction. It follows as in the proof of Lemma~\ref{Lem:meetsLieb} that each factor of ${\mathcal{P}}(z)$ meets either $\mathbf{0}_d$ or $(\mathbf{0}_{d-1},\infty)$.
\end{proof}

\begin{coro} Let $d\geq 2$ and assume that $\mathrm{gcd}(q_1,\ldots,q_d)=1$. Then the Fermi variety of $A+V$, with $A$ given by \eqref{eq:decoradj} is irreducible for all but finitely many $\lambda\in \bbC$.
\end{coro}

\begin{proof}
    If $d>2$ then this result follows from Theorem \ref{thm1}. Indeed, in this case we have that $\deg \widetilde h_0(z)=(d-1)Q$ and $\deg ((\widetilde{h}_{\infty}(\hat z))^{+})=Qd$ by \eqref{eq:h0zd} and \eqref{eq:hinftyzd}. In particular, 
    $$
    \deg \widetilde h_0(z)+ \deg ((\widetilde{h}_{\infty}(\hat z))^{+})=(2d-1)Q>(d+1)Q= \deg \widetilde{\mathcal{P}}^{+}(z),
    $$ 
    which verifies assumption \ref{assump3} as $d>2$. Moreover, Lemma \ref{lem:lowestcompdecor} implies that $p_1=p_2=1$. 
    
    For $d=2$, we have $(2d-1)Q=(d+1)Q$ 
 hence the above considerations imply assumption \ref{assumpA'3} so one only needs to check that the factorization 
$$\widetilde{\mathcal{P}}^{+}(z;\lambda)
 =K(\lambda) {\widetilde{h}}_0(z;\lambda) ({\widetilde{h}}_{\infty}(\hat{z};\lambda))^+
$$
cannot happen. In fact, this would imply that
    \begin{equation}\label{eq:setupdecorproof}
    \begin{split}
        &(z_1z_2)^Q\det(  D_z+B_V-\lambda I) \\
        & =K(\lambda)s^2(\lambda)(z_1z_2)^Qz^Q_1\left(\prod_{n \in W}r_0(\mu_n\odot z)\right)\left(\prod_{n \in W}r_{\infty}(\mu_n\odot z)\right).
\end{split}
    \end{equation}
Let us show that this can happen for at most finitely many $\lambda \in \bbC$. First, notice that $\gcd(q_1,q_2)=0$ implies that $r_0(\mu_n\odot (z_1,-z_1))$ is not identically zero for any $n\in W\setminus\{ \mathbf{0}_d\}$. In particular, setting $z_2=-z_1$ and expanding $\widetilde{\mathcal{P}}(z_1,-z_1;\lambda)$, we see that the highest power of $z_1$ in $\widetilde{\mathcal{P}}(z_1,-z_1;\lambda)$ is of the form $t(\lambda)z^{Q-1}_1$ where $t(\lambda)$ is a polynomial of degree $\nu+(\nu-1)(Q-1)$ in $\lambda$. On the other hand, if \eqref{eq:setupdecorproof} holds for some $\lambda$, using $r_0(z_1,-z_1)=\frac{1}{z_1}+\frac{1}{-z_1}=0$,  yields 
\[\widetilde{\mathcal{P}}^+(z_1,-z_1;\lambda) = 0, \ \forall z_1 \in \bbC.\]
 In particular \eqref{eq:setupdecorproof} can only be true when $t(\lambda)=0$, which in turn is true only for finitely many $\lambda$. Therefore, $\mathcal{P}(z;\lambda)$ is irreducible for all but finitely many $\lambda$.
\end{proof}

\begin{figure}

\begin{tikzpicture}[scale=3]
\def\width{2};
\def\height{3};
\def \length {2};
\foreach \m in {0,...,\width}
   \foreach \p in {0,...,\length}
        \draw[fill=black] (\m,0,\p) circle (2pt);
        \foreach \m in {0,...,\width}
         \foreach \p in {0,...,\length}
        \draw[fill=blue] (\m+0.3,0.3,\p) circle (1.7pt);
        \foreach \m in {0,...,\width}
         \foreach \p in {0,...,\length}
        \draw[fill=blue] (\m-0.3,0.3,\p) circle (1.7pt);
\foreach \m in {0,...,\numexpr\width }
    \draw[thick] (\m,0,0) -- (\m ,0, \width);
    \foreach \p in {0,...,\numexpr\length }
    \draw[thick] (\length,0,\p) -- (0 ,0, \p);
    \foreach \m in {0,...,\width}
         \foreach \p in {0,...,\length}
         \draw[thick] (\m-0.3,0.3,\p) -- (\m+0.3,0.3,\p);
          \foreach \m in {0,...,\width}
         \foreach \p in {0,...,\length}
         \draw[thick] (\m-0.3,0.3,\p) -- (\m,0,\p);
         \foreach \m in {0,...,\width}
         \foreach \p in {0,...,\length}
         \draw[thick] (\m+0.3,0.3,\p) -- (\m,0,\p);

\end{tikzpicture}

\caption{A $\bbZ^2$ decoration by triangles}
\label{fig:triangdecor}
\end{figure}
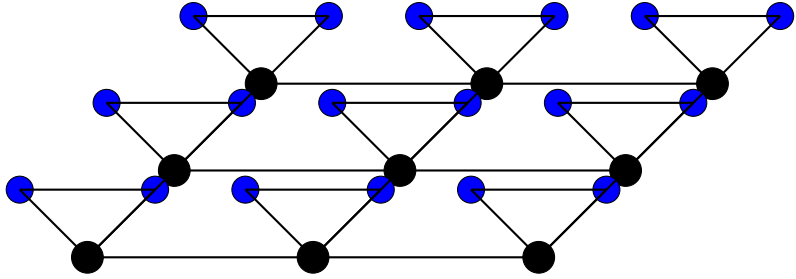

\end{document}